\documentclass[11pt,reqno]{amsart}
\usepackage{graphicx} 

\title{On a Multispecies Electrodiffusion Model in Higher-Dimensional Porous Media}

\author{Elie Abdo}
\address{	Department of Mathematics \\
     American University of Beirut \\
	Beirut 1107-2020\\Lebanon.} \email{ea94@aub.edu.lb}

    \author{Christofer Ghazale}
\address{	Department of Mathematics \\
     American University of Beirut \\
	Beirut 1107-2020\\Lebanon.} \email{cmg10@mail.aub.edu}
\usepackage{xfrac}

\usepackage[margin=1in]{geometry}
\usepackage{amsmath, amsthm, amssymb, yhmath, xfrac, mathtools}
\usepackage[most]{tcolorbox}
\usepackage[normalem]{ulem}

\usepackage{times}
\usepackage{color}
\usepackage{hyperref}

\newcommand{\la}{\label}

\newcommand{\na}{\nabla}
\newcommand{\be}{\begin{equation}}
\newcommand{\ee}{\end{equation}}
\newcommand{\ba}{\begin{array}{l}}
\newcommand{\ea}{\end{array}}

\newtheorem{thm}{Theorem}[section]
\newtheorem{prop}[thm]{Proposition}

\newtheorem{rem}[thm]{Remark}

\usepackage{cite}

\usepackage{MnSymbol,wasysym}
\usepackage{accents}

\newcommand{\R}{\mathbb R}

\date{\today}
\begin{document}

\begin{abstract}
We consider the Nernst--Planck--Darcy system describing the
electrodiffusion of $N$ ionic species with arbitrary valences and
different diffusivities, transported by a Darcy flow driven by the
electric force, in the whole space $\mathbb{R}^d$, $d \ge 3$. We prove
the existence of global solutions for initial ionic concentrations
that are small in the critical Lebesgue space
$L^{d/2}(\mathbb{R}^d)$.  Moreover, we establish
uniqueness for initial data in $L^{d+\epsilon}(\mathbb{R}^d)$ when
$3 \le d \le 4$ and in $L^{2d-1}(\mathbb{R}^d)$ when $d \ge 5$,
improving on results in the literature, which require Sobolev
regularity of the data. The proofs rely on energy estimates exploiting
cancellations of the highest-order nonlinear terms, and on the
resulting dissipative structures, which yield an instantaneous gain of
regularity used to control the doubly nonlinear velocity term in the
uniqueness argument.
\end{abstract}

\keywords{Nernst–Planck–Darcy system; electrodiffusion; global existence; uniqueness; critical spaces; dissipative structure.}

\noindent\thanks{\em{MSC Classification: 35Q35, 35A01, 35A02}}

\maketitle

\section{Introduction}
Electrodiffusion, the transport of charged particles under the
combined effects of diffusion, migration in a self-consistent electric
field, and advection by a fluid, is a fundamental mechanism in
electrochemistry \cite{lee2016membrane}, semiconductor theory \cite{biler1994debye, gajewski1986basic}, and the physiology of ion
channels and membranes \cite{gao2014high, yang2019review, davidson2016dynamical, goldman1989electrodiffusion, qian1989electro}. When the
ambient fluid is an incompressible flow through a porous medium, or
when inertial effects are negligible, the fluid velocity is governed
by Darcy's law, and the resulting model is the
Nernst--Planck--Darcy (NPD) system. In this paper we consider the NPD
system for an arbitrary number $N \ge 2$ of ionic species with
different diffusivities and different valences, posed in the whole
space $\mathbb{R}^d$, $d \ge 3$, with decay at infinity:
\begin{align}
    &\partial_t c_i + u \cdot \na c_i - D_i \Delta c_i - z_i D_i \na  \cdot (c_i \na \phi) = 0, \hspace{0.5cm} i =1, \dots, N,\label{nernst_planck} \\
    &u + \nabla p = -\rho \nabla \phi, \label{darcy_eq} \\
    &\na \cdot u = 0,\label{div_free_eq}  \\
    &-\Delta \phi = \rho = \sum\limits_{i=1}^{N} z_ic_i \label{poisson},
\end{align} 
where $c_i = c_i(x,t) \ge 0$ denotes the local concentration of the
$i$-th ionic species, $D_i > 0$ its diffusivity, $z_i \in \mathbb{R}$
its valence, $\phi$ the electric potential generated by the charge
density $\rho$ via the Poisson equation \eqref{poisson}, $u$ the
fluid velocity, and $p$ the pressure. The system is supplemented with
initial conditions
\begin{equation}\label{eq:data}
  c_i(x, 0) = c_i(0) \ge 0, \qquad i = 1, \dots, N.
\end{equation}
The Nernst--Planck equations \eqref{nernst_planck} describe the evolution of
the ionic concentrations under diffusion, electromigration, and
transport by the fluid; the velocity $u$ is determined from the
concentrations through Darcy's law \eqref{darcy_eq}, in which the
electric force $-\rho \nabla \phi$ drives the flow. In fact, the velocity is
doubly nonlinear in the concentrations, hence the
advective term $u \cdot \nabla c_i$ couples the unknowns at cubic
order. This is in contrast with the Nernst--Planck--Navier--Stokes
(NPNS) system, where the velocity is an independent dynamical
variable, and it is both a difficulty --- the nonlinearity is stronger
--- and an opportunity: the velocity inherits, with no loss, the
regularity  of the concentrations.

In this paper, we prove the existence and uniqueness of analytic solutions to the NPD system for small initial concentrations in $L^{\frac{d}{2}}(\R^d)$ and sufficiently regular initial data:

\begin{thm} \label{extension}
Let $d \ge 3$ be an integer and $\epsilon > 0$. For $i \in \left\{1, \dots, N \right\}$, let $c_i(0) \in L^{d+\epsilon}(\R^d) \cap L^1(\R^d)$ if $d = 3,4$ and   $c_i(0) \in L^{2d-1}(\R^d) \cap L^1(\R^d)$ if $d \ge 5$. Suppose $c_i(0)$ is nonnegative for all $i \in \left\{1, \dots, N \right\}$ and $\int_{\R^d} \rho_0(x) dx = 0$. Moreover, assume that each $c_i(0)$ is sufficiently small in $L^{\frac{d}{2}}(\R^d)$. Then there exists a unique analytic solution $(c_1, \dots, c_N)$ to the electrodiffusion model \eqref{nernst_planck}--\eqref{poisson} on $(0, \infty)$. 
\end{thm} 

The smallness condition  is imposed exclusively in
the scaling-critical norm (see Remark \ref{rem:criticality}); no size restriction, moment condition, or
additional regularity is required of the data. We emphasize that the
result holds for arbitrary numbers of species, arbitrary (possibly
mixed-sign) valences, and unequal diffusivities: no structural
condition such as electroneutrality of the data, equal diffusivities,
or a two-species antisymmetric configuration is assumed. This level of
generality is not merely cosmetic. Many of the techniques developed
for three-dimensional electrodiffusion systems degrade or fail outside special
configurations: energy structures based on relative entropy yield conditional global well-posedness, and cancellations
available for two species with opposite valences and equal
diffusivities are absent in the general case. 

To the best of our knowledge, all uniqueness results for the NPD
system available in the literature require initial data in Sobolev
spaces of positive regularity \cite{abdo2025long, abdo2026long, ignatova2022global}. Theorem~\ref{extension} lowers the
threshold to the Lebesgue scale. The proof relies on two mechanisms.
The first is a dissipative structure encoded in the energy estimates:
solutions gain smoothness instantaneously, beyond the regularity of
the initial data, and this gain can be quantified along the evolution
in a form that survives the passage to the difference equation. The
second is a careful analysis of the doubly nonlinear velocity term:
since $u$ is quadratic in the concentrations, the difference of the
advective terms of two solutions produces contributions in which the
difference of the unknowns is multiplied by low-regularity factors;
closing the estimates requires distributing the gained regularity
asymmetrically between the two solutions. Neither mechanism appears
to have been exploited previously in the analysis of the NPD system.

The global regularity of the solutions in the $L^{p}$ scale is
established by classical energy estimates combined with structural
cancellations which eliminate the contributions of the highest-order
nonlinear terms from the energy balance. What survives these
cancellations is not merely a closed differential inequality but a
family of dissipative structures: the resulting energy inequalities
carry additional terms from which the
solutions gain regularity beyond that imposed on the initial data,
instantaneously in time and quantitatively along the evolution. The
principal difficulty in the proof of uniqueness stems from the
velocity which, being slaved to the ionic concentrations through
Darcy's law, enters the equation for the difference of two solutions
as a doubly nonlinear term at a low level of regularity; it is
precisely the regularity gained from the dissipative structures that
permits this term to be controlled, by distributing the gain
carefully between the two solutions. The dimension plays a decisive
role in this part of the analysis. The electric potential contributes
to Darcy's law through $\nabla\phi = \nabla(-\Delta)^{-1}\rho$, a
nonlocal operator of order $-1$ acting on the charge density; on the
whole space, where no Poincar\'e inequality is available, the low
frequencies of this term cannot be absorbed by the dissipation and
must instead be handled through Hardy--Littlewood--Sobolev embeddings
and interpolation inequalities whose admissible ranges of exponents
contract as the dimension increases. This is the source of the
dimensional dichotomy in our uniqueness result: initial data in
$L^{d+\epsilon}(\mathbb{R}^d)$ suffice in dimensions $3 \le d \le 4$,
whereas in dimensions $d \ge 5$ the argument requires initial data in
$L^{2d-1}(\mathbb{R}^d)$. Finally, the solutions
constructed here are analytic in the space variables for positive
times, with quantitative lower bounds on the analyticity radius, as
can be shown by the technique of \cite{grujic1998space} for
$L^p$ data (see Remark \ref{rem:analyticity}). Whether the smallness condition in 
\eqref{extension} can be removed for large critical data, and
whether uniqueness holds in the critical space $L^{d/2}$ itself,
remain open problems that we believe are of independent interest.

In the context of the literature on
electrodiffusion models, the Nernst--Planck equations coupled with the
Navier--Stokes equations (NPNS) have been the object of intense study
in recent years. Global well-posedness and stability in two dimensions,
and conditional or small-data results in three dimensions, were
obtained in \cite{constantin2019nernst, constantin2021nernst, fischer2017global, abdo2022space, lee2021global, constantin2022nernst}; see also
\cite{abdo2026three, ignatova2021global}, for the
Nernst--Planck--Boussinesq and Nernst--Planck--Euler variants. For the
Nernst--Planck--Darcy system itself, which arises in electrokinetic
flows through porous media, the literature is more sparse: existence and
uniqueness results have been obtained in \cite{abdo2026long, herz2012existence, ignatova2022global}, as noted above, under positive Sobolev regularity of the data.

The paper is organized as follows. In Section~\ref{sec2}, we establish the global well-posedness of the NPD system. Concluding remarks are collected
in Section~\ref{sec3}.

\section{Existence and Uniqueness of Solutions}\label{sec2}

This is the main section of our paper in which we establish the global well-posedness of our $d$-dimensional NPD system. 

We need first the following elliptic regularity estimate: 

\begin{prop}\label{prop2.1}
    Let $d \ge 3$ be an integer, and let $p > d$ be a real number. Let $g$ be in $L^1(\R^d) \cap L^p(\R^d)$ such that $\int_{\R^d} g = 0$. Let $\Psi$ be a solution to the Poisson equation $- \Delta \Psi = g$  with decay at infinity. Then the following elliptic regularity estimate
    \be
\|\na \Psi\|_{L^{\infty}} \le C\left(\|g\|_{L^1} + \|g\|_{L^p} \right)
    \ee holds.
\end{prop}

\begin{proof}
    We have
    \be 
\begin{aligned}
\|\nabla\Psi\|_{L^\infty} 
&\le C(\|\nabla\Psi\|_{L^p} +\|g\|_{L^p}) 
\le C(\|(-\Delta)^{-1/2}g\|_{L^p} + \|g\|_{L^p}) 
\le C(\|g\|_{L^\frac{pd}{d+p}} + \|g\|_{L^p})
\\&\le C(\|g\|_{L^1}^\theta \|g\|_{L^p}^{1-\theta} + \|g\|_{L^p}) 
\le C\|g\|_{L^1} + C\|g\|_{L^p}
\end{aligned}
\ee where $\theta \in (0,1)$ is the parameter resulting from interpolating $L^{\frac{pd}{d+p}}$ between $L^1$ and $L^p$. Here, we made use of the Morrey, Hardy-Littlewood-Sobolev, and Young inequalities.
\end{proof}

\begin{thm} \label{main} Let $d \ge 3$ be an integer and $\epsilon > 0$. For $i \in \left\{1, \dots, N \right\}$, let $c_i(0) \in L^{d+\epsilon}(\R^d) \cap L^1(\R^d)$ if $d = 3,4$ and   $c_i(0) \in L^{2d-1}(\R^d) \cap L^1(\R^d)$ if $d \ge 5$. Suppose $c_i(0)$ is nonnegative for all $i \in \left\{1, \dots, N \right\}$ and $\int_{\R^d} \rho_0(x) dx = 0$. Moreover, assume that each $c_i(0)$ is sufficiently small in $L^{\frac{d}{2}}(\R^d)$. Then there exists a unique solution $(c_1, \dots, c_N)$ to the electrodiffusion model \eqref{nernst_planck}--\eqref{poisson} such that 
\be \la{reg1}
c_i \in L^{\infty}(0, \infty; L^{d+\epsilon}(\R^d)) \cap L^{d+1}(0, \infty; L^{d+1}(\R^{d})) \cap L^2(0, \infty; H^1(\R^d)) \cap L^d(0, \infty; L^{\frac{d^2}{d-2}}(\R^d))
\ee for any $i \in \left\{1, \dots, N \right\}$ if $d = 3,4$, and 
\be \la{reg2}
c_i \in L^{\infty}(0, \infty; L^{2d-1}(\R^d)) \cap L^{2d}(0, \infty; L^{2d}(\R^{d})) \cap L^2(0, \infty; H^1(\R^d)) 
\ee for any $i \in \left\{1, \dots, N \right\}$, if $d \ge 5$.
\end{thm}

\begin{proof}
In order to prove the global well-posedness of the model \eqref{nernst_planck}--\eqref{poisson}, we construct a regularizing system where the electric potential gradient, transport velocity, and initial data are mollified, then we establish uniform bounds in the regularization parameter and pass to the limit through applications of the Aubin-Lions lemma and the Banach-Alaoglu theorems. This type of argument is classical in the literature and will be omitted here. We focus on deriving a priori estimates.

The proof of Theorem \ref{main} is divided into two major steps.

{\bf{Step 1. $L^p$ estimates.}} Fix a real number $p > 1$. For each $i\in \{1,...,N\}$, we multiply the equation obeyed by the $i$-th ionic concentration by $c_i^{p-1}$, integrate over $\R^d$, and obtain the evolution equation 
\begin{align}\label{energyeqD}
\frac{1}{p} \frac{d}{dt} \|c_i\|_{L^p}^p + \frac{4D_i(p-1)}{p^2} \left\| \nabla c_i^{p/2} \right\|_{L^2}^2 = -\frac{2D_i z_i(p-1)}{p}\int_{\mathbb{R}^d} c_i^{p/2} \nabla\phi\cdot \nabla c_i^{p/2}\,dx.
\end{align}
Here the nonlinear term in $u$ vanishes due to the divergence-free property obeyed by $u$. 

The electromigration term satisfies a cancellation law by which the ionic concentration derivatives will not have any contribution. Indeed, integrating by parts and using the Poisson equation obeyed by the potential $\phi$, we have
\be
\int_{\mathbb{R}^d} c_i^{p/2} \nabla\phi\cdot \nabla c_i^{p/2}\,dx = \int_{\mathbb{R}^d} c_i^{p} \rho\,dx - \int_{\mathbb{R}^d} c_i^{p/2} \nabla\phi\cdot \nabla c_i^{p/2}\,dx,
\ee from which we deduce that 
\be
\int_{\mathbb{R}^d} c_i^{p/2} \nabla\phi\cdot \nabla c_i^{p/2}\,dx =\frac{1}{2}\int_{\mathbb{R}^d} c_i^{p} \rho\,dx.
\ee
Replacing $\rho$ by its law and splitting the sum, we obtain
\begin{align}
\int_{\mathbb{R}^d} c_i^{p/2} \nabla\phi\cdot \nabla c_i^{p/2}\,dx =\frac{1}{2}\int_{\mathbb{R}^d} c_i^{p/2}c_i^{p/2} \sum_{j\neq i}z_jc_j\,dx + \frac{1}{2}z_i\int_{\mathbb{R}^d}c_i^{p+1}\,dx.
\end{align}
Substituting the latter identity into \eqref{energyeqD} yields 
\begin{align}
&\frac{1}{p} \frac{d}{dt} \|c_i\|_{L^p}^p + \frac{4D_i(p-1)}{p^2} \left\| \nabla c_i^{p/2} \right\|_{L^2}^2 
\\&\quad\quad= -\frac{D_i z_i(p-1)}{p}\int_{\mathbb{R}^d} c_i^{p/2}c_i^{p/2} \sum_{j\neq i}z_jc_j\,dx -\frac{D_i z_i^2(p-1)}{p}\int_{\mathbb{R}^d}c_i^{p+1}\,dx,
\end{align}
which simplifies to 
\begin{align}\label{energyD}
&\frac{1}{p} \frac{d}{dt} \|c_i\|_{L^p}^p + \frac{4D_i(p-1)}{p^2} \left\| \nabla c_i^{p/2} \right\|_{L^2}^2 + \frac{D_i z_i^2(p-1)}{p}\|c_i\|_{L^{p+1}}^{p+1} \\&\quad\quad= -\frac{D_i z_i(p-1)}{p}\int_{\mathbb{R}^d} c_i^{p/2}c_i^{p/2} \sum_{j\neq i}z_jc_j\,dx.
\end{align}
Using Hölder's inequality with exponents
$
\frac{2d}{d-2},\;\frac{2d}{d-2},\;\frac{d}{2}$, and
the continuous Sobolev embedding of $
\dot{H}^1(\mathbb{R}^d)$ into $L^\frac{2d}{d-2}(\mathbb{R}^d)$, we have
\begin{align}\label{rhsD2}
\left| -\frac{D_i z_i(p-1)}{p}\int_{\mathbb{R}^d} c_i^{p/2}c_i^{p/2}\sum_{j\neq i} z_jc_j\,dx. \right|
&\leq \frac{CD_i z^2(p-1)}{p} \|\nabla c_i^{p/2}\|^2_{L^2}\sum_{j\neq i}\|c_j\|_{L^{d/2}},
\end{align} where $z = \max \left\{z_1, \dots, z_N\right\}.$ Denoting by $D$ the maximum of $D_1, \dots, D_N$, the latter boils down to 
\begin{align}
&\left| -\frac{D_i z_i(p-1)}{p}\int_{\mathbb{R}^d} c_i^{p/2}c_i^{p/2}\sum_{j\neq i} z_jc_j\,dx. \right|
\\&\quad\quad\leq C_{D,z,p, N}\|\nabla c_i^{p/2}\|^2_{L^2}\sum_{j\neq i} \|c_j\|_{L^{d/2}}^{d/2} + \frac{2D_i(p-1)}{p}\|\nabla c_i^{p/2}\|^2_{L^2}
\end{align} after applying Young's inequality for products.
Therefore, we infer that 
\begin{align}\label{energyD3}
\frac{1}{p} \frac{d}{dt} \|c_i\|_{L^p}^p + \frac{2D_i(p-1)}{p^2} \left\| \nabla c_i^{p/2} \right\|_{L^2}^2 + \frac{D_i z_i^2(p-1)}{p}\|c_i\|_{L^{p+1}}^{p+1} \leq C_{D, z,  p, N}\|\nabla c_i^{p/2}\|^2_{L^2}\sum_{j\neq i} \|c_j\|_{L^{d/2}}^{d/2}.
\end{align}
Summing over all indices $i\in\{1, \dots, N\}$, we obtain the differential inequality 
\be 
\begin{aligned}\label{energyD4}
&\frac{d}{dt} \sum_{i=1}^N \|c_i\|_{L^p}^p + \sum\limits_{i=1}^{N} \frac{D_i(p-1)}{p} \left\| \nabla c_i^{p/2} \right\|_{L^2}^2 + \sum\limits_{i=1}^{N} {D_i z_i^2(p-1)}\|c_i\|_{L^{p+1}}^{p+1}   
\\&\quad\quad\leq \sum_{i=1}^N \left\| \nabla c_i^{p/2} \right\|_{L^2}^2\bigg(C_{D, z,  p, N} \sum_{j = i}^N \|c_j\|_{L^{d/2}}^{d/2} - \frac{D_i(p-1)}{p}\bigg).
\end{aligned}
\ee 
Now, choosing
\begin{equation}
\delta = \frac{ \min \left\{D_1, \dots, D_N\right\}(p-1)}{p C_{D, z,  p, N}} >0 ,
\end{equation}
and assuming
\begin{equation}\label{smallness}
\sum_{j=1}^{N}\|c_j(0)\|_{L^{d/2}}^{d/2} < \delta,
\end{equation}
it follows, by continuity, that 
\begin{equation}
\sum_{j=1}^{N}\|c_j(t)\|_{L^{d/2}}^{d/2} \le \delta
\end{equation} on a short time interval $[0, T_0]$. 
Choosing $p=\frac d2$ in \eqref{energyD4}, we obtain
\begin{equation}
\frac{d}{dt} \sum_{i=1}^{N} \|c_i(t)\|_{L^{d/2}}^{d/2} \leq 0
\end{equation} on $[0,T_0]$. 
Therefore,
\begin{equation}
\sum_{i=1}^{N} \|c_i(t)\|_{L^{d/2}}^{d/2} \leq \sum_{i=1}^{N} \|c_i(0)\|_{L^{d/2}}^{d/2} <\delta
\end{equation} on $[0,T_0]$. This allows extension beyond time $T_0$ to any subsequent time $T>0$, yielding 
\begin{equation}
\sum_{j=1}^{N}\|c_i(t)\|_{L^{d/2}}^{d/2} \leq \sum_{j=1}^{N} \|c_i(0)\|_{L^{d/2}}^{d/2}
\end{equation} for any $t \ge 0$. 
Moreover, the differential inequality \eqref{energyD4} implies that 
\begin{equation} \label{eee}
\frac{d}{dt} \sum_{i=1}^{N} \|c_i(t)\|_{L^{p}}^{p} + \sum\limits_{i=1}^{N} \frac{D_i(p-1)}{p} \left\| \nabla c_i^{p/2} \right\|_{L^2}^2 + \sum\limits_{i=1}^{N} {D_i z_i^2(p-1)}\|c_i\|_{L^{p+1}}^{p+1}  \leq 0
\end{equation}  for any $p > 1$. 
Consequently, in dimensions $d= 3,4$, and at any time $t \ge 0$, we have 
\be 
\sum_{i=1}^{N} \|c_i(t)\|_{L^{d+\epsilon}}^{d+\epsilon} 
\le \sum_{i=1}^{N} \|c_i(0)\|_{L^{d+\epsilon}}^{d+\epsilon} 
\ee when $p = d + \epsilon$, 
\be 
\int_{0}^{\infty} \sum\limits_{i=1}^{N} {D_i z_i^2(d-1)}\|c_i\|_{L^{d+1}}^{d+1} dt \le \sum_{i=1}^{N} \|c_i(0)\|_{L^{d}}^{d}
\ee when $p = d$, and 
\be 
\int_{0}^{\infty} \sum\limits_{i=1}^{N} \frac{D_i}{2} \left\| \nabla c_i \right\|_{L^2}^2  dt \le \sum_{i=1}^{N} \|c_i(0)\|_{L^{2}}^{2} 
\ee when $p =2$. 
This shows that $c_i \in L^{\infty}(0,\infty, L^{d+\epsilon}(\R^d)) \cap L^{d+1}(0, \infty; L^{d+1}(\R^{d+1})) \cap L^2(0,\infty; \dot{H}^1(\R^d))$ for every $i \in \left\{1, \dots, N \right\}$. Moreover, it follows from \eqref{eee} with $p=d$ that 
\be 
\|c_i\|_{\frac{d^2}{d-2}}^{\frac{d}{2}}= \|c_i^{d/2}\|_{L^{\frac{2d}{d-2}}} \le C\|\na c_i^{d/2}\|_{L^2}
\ee due to the continuous Sobolev embedding of $\dot{H}^{1}(\R^d)$ in $L^{\frac{2d}{d-2}}(\R^d)$, and thus it follows from \eqref{eee} with $p=d$ that $c_i \in L^{d}(0, \infty; H^{\frac{d^2}{d-2}}(\R^d))$ for any $i \in \left\{1, \dots, N\right\}$. The regularity in dimensions $d  \ge 5$ also follows from $\eqref{eee}$ by taking $p = 2d-1$.

{\bf{Step 2. Uniqueness of solutions. }} Suppose $(c_1^{(1)}, \dots, c_N^{(1)})$ and $(c_1^{(2)}, \dots, c_N^{(2)})$ are two solutions of \eqref{nernst_planck}--\eqref{poisson} , both having the same initial data and satisfying the regularity criteria \eqref{reg1} if $d=3, 4$ and \eqref{reg2} if $d \ge 5$. For $i \in \left\{1, \dots, N \right\}$, we denote by $c_i, u,$ and $\phi$ the differences
\[
c_i=c_i^{(1)}-c_i^{(2)},\qquad u=u^{(1)}-u^{(2)},\qquad \phi=\phi^{(1)}-\phi^{(2)},
\]
and we let 
\[
\rho=\sum_{j=1}^N z_jc_j 
\] be the corresponding density. 
The equation obeyed by $c_i$ is given by 
\[
\partial_t c_i +u^{(1)}\cdot\nabla c_i +u\cdot\nabla c_i^{(2)} -D_i\Delta c_i -z_iD_i\nabla\cdot( c_i\nabla\phi^{(1)})- z_iD_i\nabla\cdot(c_i^{(2)}\nabla\phi)=0 .
\]
Multiplying the latter by $c_i$,  integrating over $\mathbb R^d$, using the divergence-free property obeyed by $u^{(1)}$, and summing over all indices $i \in \left\{1, \dots, N \right\}$, we obtain the energy equality
\be 
\begin{aligned}\label{fullsum}
&\frac12\frac{d}{dt}\sum_{i=1}^N\|c_i\|_{L^2}^2 +\sum_{i=1}^ND_i\|\nabla c_i\|_{L^2}^2
\\&\quad= -\sum_{i=1}^N\int_{\mathbb R^d} (u\cdot\nabla c_i^{(2)})c_i 
+\sum_{i=1}^Nz_iD_i \int_{\mathbb R^d} c_i\nabla\phi^{(1)}\cdot\nabla c_i 
 +\sum_{i=1}^Nz_iD_i \int_{\mathbb R^d} c_i^{(2)}\nabla\phi\cdot\nabla c_i .
\end{aligned}
\ee
We estimate the second term in \eqref{fullsum} using Hölder's and Young's inequality and get 
\begin{align}\label{term1bd}
\left| \sum_{i=1}^Nz_iD_i \int_{\mathbb R^d} c_i\nabla\phi^{(1)}\cdot\nabla c_i \right|
&\leq \sum_{i=1}^N\frac{D_i}{8} \|\nabla c_i\|_{L^2}^2 + C\|\nabla\phi^{(1)}\|_{L^\infty} \sum_{i=1}^N\|c_i\|_{L^2}^2.
\end{align}
Moreover, applying the elliptic regularity estimate  \eqref{prop2.1} with $p = d+\epsilon$, we can further bound $\nabla\phi^{(1)}$ in $L^{\infty}$ by
\begin{align}\label{phi1bd}
\|\nabla\phi^{(1)}\|_{L^\infty}
&\leq C\left(\sum_{j=1}^N\|c_j^{(1)}\|_{L^1} + \sum_{j=1}^N\|c_j^{(1)}\|_{L^{d+\varepsilon}} \right),
\end{align} yielding 
\begin{align}\label{term2}
\left| \sum_{i=1}^Nz_iD_i \int_{\mathbb R^d} c_i\nabla\phi^{(1)}\cdot\nabla c_i \right|
&\leq \frac{D}{8} \sum_{i=1}^N\|\nabla c_i\|_{L^2}^2 + C \left(\sum_{j=1}^N\|c_j^{(1)}\|_{L^1} + \sum_{j=1}^N\|c_j^{(1)}\|_{L^{d+\varepsilon}} \right) \sum_{i=1}^N\|c_i\|_{L^2}^2 .
\end{align}
For the third term in \eqref{fullsum}, we use Hölder's inequality with exponents
$
\frac1d, \frac{d-2}{2d}, \frac12$, together with the Hardy-Littlewood-Sobolev inequality
\be 
\|\nabla\phi\|_{L^{\frac{2d}{d-2}}} \leq C \|\rho\|_{L^2},
\ee  to estimate 
\be \label{term3}
\begin{aligned}
&\left|
\sum_{i=1}^Nz_iD_i
\int_{\mathbb R^d}
c_i^{(2)}
\nabla\phi\cdot\nabla c_i
\right|
\leq
C\sum_{i=1}^N
\|c_i^{(2)}\|_{L^d}
\|\nabla\phi\|_{L^{\frac{2d}{d-2}}}
\|\nabla c_i\|_{L^2}
\leq C\sum_{i=1}^N \|c_i^{(2)}\|_{L^d} \|\rho\|_{L^2} \|\nabla c_i\|_{L^2} 
\\&\leq\frac{D}{8} \sum_{i=1}^N \|\nabla c_i\|_{L^2}^2 + C  \|\rho\|_{L^{2}}^2 \sum_{i=1}^N \|c_i^{(2)}\|_{L^d}^2
\le \frac D8 \sum_{i=1}^N \|\nabla c_i\|_{L^2}^2 + C  \sum_{i=1}^N\|c_i\|_{L^{2}}^2 \sum_{j=1}^N \|c_j^{(2)}\|_{L^d}^2.
\end{aligned}
\ee The first term on the right-hand-side of \eqref{fullsum} is the most challenging one due to the doubly nonlinear structure of $u$. Integrating by parts, applyig H\"older's inequality with exponents $q, p, 2$ with $q > d$ and $p \in (2, \frac{2d}{d-2})$, and using the boundedness of the Leray projector on $L^p$, we have   
\be 
\begin{aligned}
&\left| \int ( u\cdot\nabla c_i^{(2)})  c_i \right|
= \left| \int c_i^{(2)} u\cdot\nabla c_i \right| \le \| c_i^{(2)}\|_{L^q} \|u\|_{L^p} \| \nabla c_i\|_{L^2} 
\\&\le C \|c_i^{(2)}\|_{L^q}\left(\|\rho\nabla\phi^{(1)}\|_{L^p} + \|\rho^{(2)}\nabla\phi\|_{L^p} \right) \|\nabla c_i\|_{L^2}\\
&\le C \| c_i^{(2)}\|_{L^q} \|\rho\|_{L^p} \|\nabla\phi^{(1)}\|_{L^\infty} \|\nabla c_i\|_{L^2}
+ C \|c_i^{(2)}\|_{L^q} \|\rho^{(2)}\nabla\phi\|_{L^p} \|\nabla c_i\|_{L^2} .
\end{aligned}
\ee We denote the first product by $\mathcal{I}_{i,1}$ and the second one by $\mathcal{I}_{i,2}$, and we bound each one of them separately. For $\mathcal{I}_{i,1}$, we interpolate and employ elliptic estimates to estimate $\mathcal{I}_{i,1}$ as follows,
\be 
\begin{aligned}
\mathcal{I}_{i,1} &= C\|c_i^{(2)}\|_{L^q} \|\rho\|_{L^p} \|\nabla\phi^{(1)}\|_{L^\infty} \|\nabla c_i\|_{L^2}  \\ 
 &\le C\| c_i^{(2)}\|_{L^q} \|\rho\|_{L^2}^{1-\theta} \|\nabla \rho\|_{L^2}^{\theta} (\|\rho^{(1)}\|_{L^1} + \|\rho^{(1)}\|_{L^{d+\epsilon}}) \|\nabla c_i\|_{L^2} 
  \\&\le C \|c_i^{(2)}\|_{L^q} \left(\sum_{j=1}^N\|c_j\|_{L^2}^2 \right)^{\frac{1-\theta}{2}} \left(\sum_{j=1}^N\|\nabla c_j\|_{L^2}^2 \right)^{\frac{\theta}{2}} \left(\|\rho^{(1)}\|_{L^1} + \|\rho^{(1)}\|_{L^{d+\epsilon}} \right) \|\nabla c_i\|_{L^2}
\end{aligned}
\ee for some $\theta \in (0,1)$.
Summing over all indices $i \in \left\{1, \dots, N \right\}$ and applying the  Cauchy-Schwarz inequality, we deduce that 
\be
\begin{aligned}
&\sum_{i=1}^N \mathcal{I}_{i,1} 
\\&\le C \left(\sum_{i=1}^N \|c_i^{(2)}\|_{L^q}^2 \right)^{1/2} \left(\sum_{i=1}^N \|\nabla c_i\|_{L^2}^2 \right)^{1/2} \left(\sum_{j=1}^N\|c_j\|_{L^2}^2 \right)^{\frac{1-\theta}{2}} \left(\sum_{j=1}^N\|\nabla c_j\|_{L^2}^2
\right)^{\frac{\theta}{2}} \left(\|\rho^{(1)}\|_{L^1} + \|\rho^{(1)}\|_{L^{d+\epsilon}}\right),
\end{aligned}
\ee which, after applying Young's inequality and taking $q = d+\epsilon$, reduces to
\begin{align}\label{term1.1}
\sum_{i=1}^N \mathcal{I}_{i,1} &\le \frac D8 \sum_{i=1}^N \|\nabla c_i\|_{L^2}^2 + C A(t)^{\frac{2}{1-\theta}} \sum_{i=1}^N \|c_i\|_{L^2}^2 .
\end{align}
where
\[
A(t)= \left(\sum_{i=1}^N \|c_i^{(2)}\|_{L^{d+\epsilon}}^2 \right)^{1/2} (\|\rho^{(1)}\|_{L^1} + \|\rho^{(1)}\|_{L^{d+\epsilon}}),
\] which is integrable in time. Now we turn our attention to the term $\mathcal{I}_{i,2}$, which can be bounded by
\begin{align}\label{term1.2}
\mathcal{I}_{i,2} \le C \|c_i^{(2)}\|_{L^{2d}} \|\rho^{(2)}\|_{L^{2d}} \|\nabla\phi\|_{L^{\frac{2d}{d-2}}} \|\nabla c_i\|_{L^2}. 
\end{align} This term depends on the spatial dimension as follows: 

\textbf{Case 1: $d=3$.} In this case, we interpolate in $L^p$ spaces and employ continuous Sobolev embeddings to estimate $\mathcal{I}_{i,2}$ by 
\be 
\begin{aligned}
\mathcal{I}_{i,2} &\le \|c_i^{(2)}\|_{L^{6}} \|\rho^{(2)}\|_{L^{6}} \|\nabla\phi\|_{L^6} \|\nabla c_i\|_{L^2}
\\&\le C \|c_i^{(2)}\|_{L^{3}}^\frac{1}{4} \|c_i^{(2)}\|_{L^{9}}^\frac{3}{4} \|\rho^{(2)}\|_{L^{3}}^\frac{1}{4} \|\rho^{(2)}\|_{L^{9}}^\frac{3}{4} \|\rho\|_{L^{2}} \|\nabla c_i\|_{L^2}  \\
&\le C \|c_i^{(2)}\|_{L^{3}}^\frac{1}{4} \|c_i^{(2)}\|_{L^{9}}^\frac{3}{4} \bigg(\sum_{j=1}^N
\|c_j^{(2)}\|_{L^3}^2\bigg)^{1/8} \bigg(\sum_{j=1}^N
\|c_j^{(2)}\|_{L^9}^2\bigg)^{3/8} \bigg(\sum_{j=1}^N
\|c_j\|_{L^2}^2\bigg)^{1/2} \|\nabla c_i\|_{L^2} \\
&\le \frac{D}{8N} \|\nabla c_i\|_{L^2}^2 + C \|c_i^{(2)}\|_{L^{3}}^\frac{1}{2} \|c_i^{(2)}\|_{L^{9}}^\frac{3}{2} \bigg(\sum_{j=1}^N \|c_j^{(2)}\|_{L^3}^2\bigg)^{1/4} \bigg(\sum_{j=1}^N \|c_j^{(2)}\|_{L^9}^2\bigg)^{3/4} \sum_{j=1}^N \|c_j\|_{L^2}^2.
\end{aligned}
\ee 
Summing and simplifying, we infer that 
\be
\begin{aligned}\label{t1case1}
\sum\limits_{i=1}^{N}  \mathcal{I}_{i,2} &\le \frac{D}{8N} \sum_{i=1}^N \|\nabla c_i\|_{L^2}^2 + C \sum_{i=1}^N \|c_i\|_{L^2}^2 \bigg(\sum_{j=1}^N \|c_j^{(2)}\|_{L^3}^2\bigg)^{1/2} \bigg(\sum_{j=1}^N \|c_j^{(2)}\|_{L^9}^2\bigg)^{3/2}.
\end{aligned}
\ee Here we exploit the additional regularity $c_i \in L^3(0, \infty; L^9(\R^3))$ from \eqref{reg1} to guarantee integrability in time of the last term.

\textbf{Case 2: $d=4$. } This case is similar to the previous one, where the additional regularity $c_i \in L^4(0, \infty; L^8(\R^4))$ can be taken advantage of. Indeed, it holds that 
\be 
\begin{aligned}
\mathcal{I}_{i,2} &\le C\|c_i^{(2)}\|_{L^{8}} \|\rho^{(2)}\|_{L^{8}} \|\nabla\phi\|_{L^4} \|\nabla c_i\|_{L^2}
\\&\le C\|c_i^{(2)}\|_{L^{8}} \|\rho^{(2)}\|_{L^{8}} \|\rho\|_{L^2} \|\nabla c_i\|_{L^2} \nonumber \\
&\le \frac{D}{8N} \|\nabla c_i\|_{L^2}^2 + C \|\rho\|_{L^2}^2 \|c_i^{(2)}\|_{L^{8}}^2 \|\rho^{(2)}\|_{L^{8}}^2,
\end{aligned}
\ee which, after summing over all indices $i \in \left\{1, \dots, N \right\}$, reduces to 
\begin{align}\label{t1case2}
\sum\limits_{i=1}^{N} \mathcal{I}_{i,2} &\le \frac D8 \sum_{i=1}^N \|\nabla c_i\|_{L^2}^2 + C \sum_{i=1}^N \|c_i\|_{L^2}^2 \bigg(\sum_{j=1}^N \|c_j^{(2)}\|_{L^{8}}^2\bigg)^2.
\end{align}
We point out that for cases 1 and 2, we only need the initial concentrations to be in  $L^{d+\epsilon} (\mathbb{R}^d)).$

\textbf{Case 3: $d\ge 5$.} 
The best norm division we can do here is 
\be 
\mathcal{I}_{i,2} \le C\|c_i^{(2)}\|_{L^{2d}} \|\rho^{(2)}\|_{L^{2d}} \|\nabla\phi\|_{L^{\frac{2d}{d-2}}} \|\nabla c_i\|_{L^2},
\ee
and for this we would need better regularity assumptions on $c_i(0)$, namely, $c_i(0) \in L^\infty(0,T,L^{2d-1}(\mathbb{R}^d)).$ In fact, using the embedding of $\dot{H}^1$ in $L^{\frac{2d}{d-2}}$, we can bound the sum of all the $\mathcal{I}_{i,2}$ terms by 
\be 
\begin{aligned}\label{t1case3}
\sum\limits_{i=1}^{N} \mathcal{I}_{i,2} \le \frac D8 \sum_{i=1}^N \|\nabla c_i\|_{L^2}^2 + C \sum_{i=1}^N\|c_i\|_{L^2}^2 \bigg(\sum_{j=1}^N \|c_j^{(2)}\|_{L^{2d}}^2 \bigg)^2 .
\end{aligned}
\ee The integrability in time of the fourth powers of the $L^{2d}$ norms of the ionic concentrations hold due to the regularity criteria $c_i \in L^{2d}(0, \infty; L^{2d}(\R^d))$ stated in \eqref{reg2}.

{\bf{Conclusion.}} In all of the above three cases, we end up with a differential inequality of the form 
\be 
\begin{aligned}
\frac{d}{dt}\sum_{i=1}^N \|c_i\|_{L^2}^2 & \le  C G(t) \sum_{i=1}^N\|c_i\|_{L^2}^2,
\end{aligned}
\ee where $G(t)$ is an integrable function of time over $[0,T]$. 
By Gronwall's inequality, we deduce that $c_i^{(1)} = c_i^{(2)}$ a.e. 
\end{proof}

\section{Concluding Remarks}\label{sec3}

We close the paper with several remarks concerning the scope of our
methods, the relation of our results to the existing literature, and
some further properties of the solutions constructed above.

\begin{rem}[Applicability to other electrodiffusion models]
\label{rem:other-models}
The proof of Theorem~\ref{main} 
exploits the dissipative structure of the Nernst--Planck subsystem
together with the fact that, in the Darcy regime, the fluid velocity
is determined instantaneously from the ionic concentrations at the
same level of regularity. This structure is shared, in whole or in
part, by a number of electrodiffusion models, and our arguments apply
to them with only minor modifications. In particular, for the
three-dimensional Nernst--Planck equations (in the absence of fluid
coupling), both the global existence of solutions for initial
concentrations that are small in $L^{d/2}(\mathbb{R}^d)$ and the
uniqueness argument carry over
verbatim. For the three-dimensional Nernst--Planck--Boussinesq (NPB)
and Nernst--Planck--Navier--Stokes (NPNS) systems, the existence part
of our analysis persists: for ionic concentrations with small initial
data in $L^{d/2}(\mathbb{R}^d)$, one obtains global solutions with the
concentrations lying in the same $L^p$-based spaces as in
Theorem~\ref{main}. The uniqueness part, however, does not
carry over to these fluid couplings, as it would require, in
particular, uniqueness at supercritical regularity levels for the
three-dimensional Navier--Stokes equations, which is not available.
This contrast highlights a genuine advantage of the Darcy coupling:
the velocity field inherits the (low) regularity of the ionic
concentrations without loss, which is precisely what makes uniqueness
in scaling-critical and nearly critical spaces tractable.
\end{rem}

\begin{rem}[Comparison with previous uniqueness results]
\label{rem:comparison}
Our uniqueness result improves upon the results available in the
literature for the three-dimensional NPD system, where uniqueness has
been established under the assumption that the initial ionic
concentrations belong to Sobolev spaces of positive regularity
\cite{ abdo2025long, abdo2026long, ignatova2022global}. In contrast,
Theorem~\ref{main} yields uniqueness for merely
Lebesgue-integrable initial data, namely
$c_i(0) \in L^{d+\epsilon}(\mathbb{R}^d)$ for $3 \le d \le 4$ and
$c_i(0) \in L^{2d-1}(\mathbb{R}^d)$ for $d \ge 5$, which lie strictly
below the regularity thresholds previously required. The improvement
rests on two ingredients. The first is a dissipative structure encoded
in the energy estimates: for positive times, the solutions
instantaneously gain smoothness beyond the regularity imposed on the
initial data, and this parabolic gain, quantified along the evolution,
provides exactly the amount of integrability needed to control the
difference of two solutions emanating from the same low-regularity
data. The second is a careful analysis of the velocity term which,
through the Darcy law and the elliptic coupling with the electric
potential, is doubly nonlinear in the ionic concentrations: the
velocity itself is quadratic in the concentrations, so that the
electromigration--advection contribution to the difference equation
produces terms in which the unknowns appear at second order. Handling
these terms requires estimates in which the gained regularity of the
two solutions is distributed asymmetrically between the factors, a
mechanism that does not appear to have been exploited previously in
the analysis of the NPD system and which is precisely what allows the
uniqueness argument to close at this level of regularity.
\end{rem}

\begin{rem}[Spatial analyticity]
\label{rem:analyticity}
Under the assumptions of Theorem~\ref{main}, the global
solutions constructed in this paper are in fact analytic in the space
variables for all positive times. This follows from the method of
Grujić and Kukavica \cite{grujic1998space}, originally developed to establish spatial
analyticity for the Navier--Stokes equations with $L^p$ initial data,
which is well suited to the mild-solution framework adopted here: the
quadratic and drift nonlinearities of the NPD system obey product
estimates in the analytic norms of of the
same type as the Navier--Stokes nonlinearity, and the Darcy velocity
is controlled by the concentrations through operators of
Calder\'on--Zygmund type, which act boundedly on the relevant analytic
classes. 
\end{rem}

\begin{rem}[Criticality of the smallness assumption]
\label{rem:criticality}
The smallness assumption in Theorem~\ref{main} is imposed in
the space $L^{d/2}(\mathbb{R}^d)$, which is critical with respect to
the natural scaling of the NPD system.
The criticality of $L^{d/2}(\mathbb{R}^d)$ can be seen as follows. If
$(c_1,\dots,c_N,\phi,u)$ is a solution of the system, consider, for
$\lambda>0$, the rescaled functions
\be\label{eq:rescaled}
c_i^{\lambda}(x,t)=\lambda^{2}c_i(\lambda x,\lambda^{2}t),\qquad
\phi^{\lambda}(x,t)=\phi(\lambda x,\lambda^{2}t),\qquad
u^{\lambda}(x,t)=\lambda\,u(\lambda x,\lambda^{2}t).
\ee
Writing $y=\lambda x$, $s=\lambda^{2}t$ and
$\rho^{\lambda}=\sum_j z_j c_j^{\lambda}=\lambda^{2}\rho(y,s)$, the
chain rule gives
\be
-\Delta_x\phi^{\lambda}=-\lambda^{2}(\Delta_y\phi)(y,s)
=\lambda^{2}\rho(y,s)=\rho^{\lambda},
\ee so the Poisson equation is
preserved, while
\be 
\partial_t c_i^{\lambda}=\lambda^{4}(\partial_s c_i)(y,s),
\Delta_x c_i^{\lambda}=\lambda^{4}(\Delta_y c_i)(y,s),
\mathrm{div}_x(c_i^{\lambda}\nabla_x\phi^{\lambda})
=\lambda^{4}\big(\mathrm{div}_y(c_i\nabla_y\phi)\big)(y,s)
\ee and 
\be 
u^{\lambda}\cdot\nabla_x c_i^{\lambda}
=\lambda^{4}(u\cdot\nabla_y c_i)(y,s),
\ee so that every term of the
Nernst--Planck equations picks up the same factor $\lambda^{4}$ and
$(c_1^{\lambda},\dots,c_N^{\lambda},\phi^{\lambda},u^{\lambda})$ is
again a solution, with initial data
$\lambda^{2}c_i^{0}(\lambda\,\cdot)$. The exponents in
\eqref{eq:rescaled} are the only possible ones: matching the time
derivative with the diffusion imposes the parabolic scaling of time,
matching the quadratic electromigration term
$\mathrm{div}(c_i\nabla\phi)$ with the diffusion, in combination with
the Poisson equation, forces the exponent $2$ on the concentrations
and the invariance of $\phi$. Since a change of variables yields
\be\label{eq:norm-scaling}
\|c_i^{\lambda}(0)\|_{L^{p}(\mathbb{R}^d)}
=\lambda^{2-\frac{d}{p}}\|c_i(0)\|_{L^{p}(\mathbb{R}^d)},
\ee
the exponent $2-d/p$ vanishes if and only if $p=d/2$, so that
$L^{d/2}(\mathbb{R}^d)$ is the unique Lebesgue space whose norm is
invariant under \eqref{eq:rescaled}: spaces with $p>d/2$ are
subcritical and spaces with $p<d/2$ are supercritical. In particular,
the smallness condition of Theorem~\ref{main} constrains only
the size of the initial data and not their spatial scale.
 Whether the smallness
condition can be removed for large critical data, in dimensions
$d \ge 3$ and for arbitrary numbers of ionic species with different
diffusivities and valences, remains an interesting open problem.
\end{rem}

\vspace{0.5cm}

{\bf{Acknowledgments.}} E.A. and C.G. were partially supported by the University Research Board (URB) of the American University of Beirut under Grant No. 104752. 

\vspace{0.5cm}

{\bf{Data Availability Statement.}} The research does not have any associated data.

\vspace{0.5cm}

{\bf{Conflict of Interest.}} The authors declare that they have no conflict of interest.

\bibliographystyle{plain}
\bibliography{references}

\end{document}